\documentclass{amsart}
\usepackage[english]{babel}
\usepackage[latin1]{inputenc}
 \usepackage[all]{xy}

\usepackage{amsmath,amsfonts,amssymb,amsthm,amscd,array,stmaryrd,mathrsfs, mathdots}
\PassOptionsToPackage{option}{xcolor}
\usepackage{pstricks}

\theoremstyle{plain}
\newtheorem{thm}{Theorem}
\newtheorem{lem}{Lemma}[section]

\newtheorem{prop}[lem]{Proposition}

\theoremstyle{definition}

\newtheorem{rem}[lem]{Remark}
\newtheorem{ex}[lem]{Example}

\newcommand{\R}{\mathbb{R}}
\newcommand{\Z}{\mathbb{Z}}

\newcommand{\N}{\mathbb{N}}

\newcommand{\Q}{\mathbb{Q}}

\newcommand{\Rc}{\mathcal{R}}
\newcommand{\Sc}{\mathcal{S}}

\begin{document}

\title[$q$-deformed real numbers]{On $q$-deformed real numbers}

\author{Sophie Morier-Genoud, Valentin Ovsienko}

\address{Sophie Morier-Genoud,
Sorbonne Universit\'e, Universit\'e Paris Diderot, CNRS,
Institut de Math\'ematiques de Jussieu-Paris Rive Gauche,
 F-75005, Paris, France
}

\address{
Valentin Ovsienko,
Centre national de la recherche scientifique,
Laboratoire de Math\'ematiques 
U.F.R. Sciences Exactes et Naturelles 
Moulin de la Housse - BP 1039 
51687 Reims cedex 2,
France}
\email{sophie.morier-genoud@imj-prg.fr, valentin.ovsienko@univ-reims.fr}

\keywords{$q$-rationals, $q$-continued fractions, total positivity, Farey tessellation.}


\begin{abstract}
We associate a formal power series with integer coefficients to a positive real number,
we interpret this series as a ``$q$-analogue of a real.''
The construction is based on the notion of $q$-deformed rational number introduced in arXiv:1812.00170.
Extending the construction to negative real numbers, we obtain certain Laurent series.
\end{abstract}

\maketitle

\thispagestyle{empty}


\section{Introduction}\label{IntSec}

We take a new and experimental route to introduce a certain version of ``$q$-deformed real numbers'',
extending $q$-deformations of rationals introduced in~\cite{MGOqR}.
Given a real number~$x\geq0$, we will construct a formal power series with integer coefficients associated with~$x$.
There is no explicit formula to determine the coefficients of this series,
but there is an algorithm to calculate them.

Our construction is completely different from the classically known $q$-deformations of a number~$x\in\R$,
defined by the formulas 
$\frac{q^x-1}{q-1}$ (or $\frac{q^x-q^{-x}}{q-q^{-1}}$) that do not give power series with integer coefficients.
The only case where our construction coincides with the classical one is that of integers.
Throughout this paper, we always use the Gauss definition:
$$
\left[a\right]_{q}=\frac{q^a-1}{q-1},
\qquad\qquad
a\in\Z,
$$
that gives the familiar polynomials
$\left[a\right]_{q}=1+q+\cdots+q^{a-1}$
and $\left[-a\right]_{q}=-q^{-1}-q^{-2}-\cdots-q^{-a}$, for~$a\in\N$.
All our constructions will be in accordance with these formulas.

It is of course too early to discuss possible applications of the $q$-deformed real numbers,
since we do not know sufficiently general properties of the power series we obtain.
However, the existence of the procedure is quite surprising and the examples are captivating.

The $q$-deformation of a rational number~$\frac{r}{s}$
is a quotient of two polynomials:
$$
\left[\frac{r}{s}\right]_{q}=\frac{\Rc(q)}{\Sc(q)},
$$
where~$\Rc$ and~$\Sc$ both depend on~$r$ and~$s$ (see~\cite{MGOqR}).
In this paper, we represent these rational functions as Taylor series at $q=0$.

The definition of $q$-reals is as follows.
Let first $x\geq1$ be an irrational number, and~$(x_n)_{n\geq1}$ any sequence of rational numbers that converges to~$x$.
We $q$-deform the sequence~$(x_n)_{n\geq1}$ to 
obtain a sequence of rational functions:
$\left[x_1\right]_q,\left[x_2\right]_q,\left[x_3\right]_q,\ldots$
For every $n\geq1$, we
consider the Taylor expansion of the rational function~$\left[x_n\right]_q$ at~$q=0$:
\begin{equation}
\label{TSEq}
\left[x_n\right]_q=:\sum_{k\geq0}\varkappa_{n,k}\,q^k.
\end{equation}
Abusing the notation, we use the same name, $\left[x_n\right]_q$, for the Taylor series.
The $q$-deformation of~$x$ is the series
\begin{equation}
\label{TSEqBis}
\left[x\right]_q:=\sum_{k\geq0}\varkappa_k\,q^k,
\qquad\hbox{where}\qquad 
\varkappa_k=\lim_{n\to\infty}\varkappa_{n,k}.
\end{equation}
The existence of the limit and its independence of the choice of the converging sequence~$(x_n)_{n\geq1}$
is guaranteed by the following theorem.

\begin{thm}
\label{ConvThm}
Given an irrational real number $x\geq1$, for every~$k\geq0$ the coefficients~$\varkappa_{n,k}$ of the Taylor series~\eqref{TSEq} stabilize as~$n$ grows.
Moreover, the limit coefficients~$\varkappa_k$ in~\eqref{TSEqBis} are integers that do not depend on the choice of the converging sequence of rationals $(x_n)_{n\geq1}$.
\end{thm}

This statement was first observed by computer experimentation, but then a simple
proof was found.
Note that the coefficients of the polynomials in the numerator and denominator of 
the sequence of rational functions $\left[x_n\right]_q$ do not stabilize.
They grow with~$n$ infinitely at every fixed power of~$q$.

In practice, we always construct $q$-deformed real numbers using continued fractions.
Let $x\geq1$ be a real number, and
$x=\left[a_1,a_2,a_3,\ldots\right],$
where~$a_i$'s are positive integers,
its continued fraction expansion.
The sequence of rational numbers  
$$
x_n:=\left[a_1,\ldots , a_n\right]
$$
approximates~$x$;
it is called the sequence of convergents.
In this case the stabilization phenomenon
of Theorem~\ref{ConvThm} can be controlled with a greater exactness.

\begin{prop}
\label{TechLem}
Let $x\geq1$ be an irrational real number.
The Taylor expansions at $q=0$ of two consecutive $q$-deformed
convergents of the continued fraction of $x$, 
namely of $x_{n-1}=[a_1,\ldots,a_{n-1}]$ and $x_n=[a_1,\ldots,a_n]$, have the first $a_1+\cdots+a_n-1$ terms identical,
the coefficients of $q^{a_1+\cdots+a_n-1}$ differ by~$1$.
\end{prop}

With the help of a computer program, we carried out a number of tests
calculating $q$-deformations of known mathematical constants,
from the simplest golden ratio to the transcendental $e$ and~$\pi$, and checking their properties.
The most pleasant surprise for us was the appearance of the sequence
of generalized Catalan numbers (sequence A004148 of \cite{OEIS})
as coefficients~$\varkappa_k$ of the deformed golden ratio.
This remarkable ``coincidence'' and known properties of A004148 allowed us to conjecture
(and eventually prove) several properties of quadratic irrationals.

We know only very few general properties of $q$-deformed real numbers.
One of them is the action of the translation group~$\Z$ described by following formula
\begin{equation}
\label{TransEq}
\left[x+1\right]_q=q\left[x\right]_q+1,
\qquad\qquad
\left[x-1\right]_q:=\frac{\left[x\right]_q-1}{q}.
\end{equation}
Note that the second equation in~\eqref{TransEq} allows us to extend our $q$-deformations
to the case $x<1$ (including negative real numbers); 
the notation ``$:=$'' means that we use this equation as definition.

The property~\eqref{TransEq} implies the following ``gap theorem''.
\begin{thm}
\label{GapThm}
If $k\leq{}x\leq{}k+1$, where $k\in\Z_{>0}$, then the $k$-th order coefficient of the series~$\left[x\right]_q$ vanishes,
while all the preceding coefficients are equal to~$1$:
$$
\left[x\right]_q=1+q+q^2+\cdots+q^{k-1}+\varkappa_{k+1}\,q^{k+1}+\varkappa_{k+2}\,q^{k+2}+\cdots
$$
\end{thm}

Theorem~\ref{GapThm} implies that for negative~$x$ we obtain Laurent series instead of power series.
More precisely, if $-k\leq{}x<1-k$, where $k\in\Z_{>0}$, then~$\left[x\right]_q$ is of the following general form
\begin{equation}
\label{GeneralEq}
\left[x\right]_q=-q^{-k}+\varkappa_{1-k}\,q^{1-k}+\varkappa_{2-k}\,q^{2-k}+\cdots
\end{equation}
where $\varkappa_i\in\Z$.

Let us sum up our understanding of $q$-deformation, or ``quantization'', 
of real numbers in comparison with integers and rationals.
This quantization transforms integers into polynomials,
rationals into rational fractions, and real numbers into power series, in each case with integer coefficients:
$$
{\left[{\,.\,}\right]_q}\;:
\left\{
\begin{array}{rcl}
\Z_{\geq1}&\longrightarrow&\Z_{\geq0}[q],\\[6pt]
\Q_{\geq0}&{\longrightarrow}&\Z_{\geq0}(q),\\[6pt]
\R_{\geq0}&\longrightarrow&\Z[[q]],\\[6pt]
\R&\longrightarrow&\Z[[q]][q^{-1}].
\end{array}
\right.
$$
In the case of integers and rationals, the resulting polynomials and rational functions are with positive coefficients.
In the case of real numbers, this positivity consists in the fact that the power series are obtained
as limit of rationals functions with positive integer coefficients.

The paper is organized as follows.

In Section~\ref{ICFqSec} we briefly recall the notion of $q$-rational.
We start with the most elementary, recurrent way to calculate $q$-rationals,
and then give two equivalent and more explicit formulas.
The first one uses the continued fractions and the second $2\times2$ matrices.
The reader can find several other equivalent definitions in~\cite{MGOqR}.

In Section~\ref{ProoSec} we prove Theorem~\ref{ConvThm} and Proposition~\ref{TechLem}.

In Section~\ref{ExSec}
we use a computer program to
investigate several examples of $q$-deformed quadratic irrational numbers.
We start with the golden ratio and continue with the ``silver ratio'' and several examples of square roots,
giving in each case an explicit formula of the $q$-deformation.

In Section~\ref{ExSecBis} we consider two
examples of transcendental irrationals, namely~$e$ and~$\pi$.
We calculate the first terms of their $q$-deformations trying to make some observations.

In Section~\ref{TransSec} we discuss the action of the translation group and prove Theorem~\ref{GapThm}.

It would be interesting to investigate more concrete examples.
Note that we searched in vain for some functional equations similar to those of $q$-deformed quadratic irrational numbers
in the case of higher order algebraic numbers,
such as $\sqrt[\leftroot{-2}\uproot{2}3]{2}$.

\section{$q$-deformed rationals}\label{ICFqSec}

In this section,
we try to give a transparent and self-contained exposition of the notion of $q$-rational introduced in~\cite{MGOqR}.
We outline an analogy with $q$-binomial coefficients,
and give a recurrent way to compute $q$-rationals from the Farey graph.
Finally we give two explicit formulas for the $q$-rationals using continued fraction expansions and the matrix form.

\subsection{Analogy with the $q$-binomials}

Recall that the classical Gaussian $q$-binomial coefficients (see~\cite{Sta}) are polynomials in~$q$ that 
can be calculated recurrently via the formula
\begin{equation}
\label{PascalEq}
{r\choose s}_q=
{r-1\choose s-1}_q+q^s{r-1\choose s}_q.
\end{equation}
The $q$-binomial coefficients are the vertices of the ``weighted Pascal triangle'' which encodes the above formula:
$$
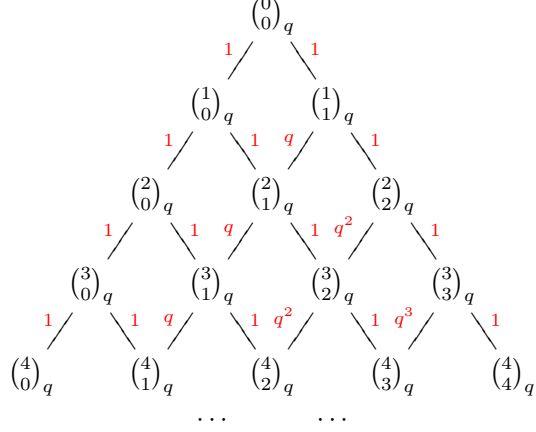
\xymatrix @!0 @R=0.6cm @C=0.8cm
{
&&&&{0\choose 0}_q\ar@{-}[rdd]^{\textcolor{red}{1}}\ar@{-}[ldd]_{\textcolor{red}{1}}
\\
\\
&&&{1\choose 0}_q\ar@{-}[rdd]^{\textcolor{red}{1}}\ar@{-}[ldd]_{\textcolor{red}{1}}
&&{1\choose 1}_q\ar@{-}[rdd]^{\textcolor{red}{1}}\ar@{-}[ldd]_{\textcolor{red}{q}}
\\
\\
&&{2\choose 0}_q\ar@{-}[rdd]^{\textcolor{red}{1}}\ar@{-}[ldd]_{\textcolor{red}{1}}
&&{2\choose 1}_q\ar@{-}[rdd]^{\textcolor{red}{1}}\ar@{-}[ldd]_{\textcolor{red}{q}}&&{2\choose 2}_q\ar@{-}[rdd]^{\textcolor{red}{1}}\ar@{-}[ldd]_{\textcolor{red}{q^2}}
\\
\\
&{3\choose 0}_q\ar@{-}[rdd]^{\textcolor{red}{1}}\ar@{-}[ldd]_{\textcolor{red}{1}}
&&{3\choose 1}_q\ar@{-}[rdd]^{\textcolor{red}{1}}\ar@{-}[ldd]_{\textcolor{red}{q}}
&&{3\choose 2}_q\ar@{-}[rdd]^{\textcolor{red}{1}}\ar@{-}[ldd]_{\textcolor{red}{q^2}}
&&{3\choose 3}_q\ar@{-}[rdd]^{\textcolor{red}{1}}\ar@{-}[ldd]_{\textcolor{red}{q^3}}
\\
\\
{4\choose 0}_q
&&{4\choose 1}_q&&{4\choose 2}_q&&{4\choose 3}_q&&{4\choose 4}_q
\\ 
&&&\cdots&&\cdots
}
$$
The idea behind the definition of $q$-rationals is to use exactly the same rule,
but replace the Pascal triangle by the Farey graph.

\subsection{The weighted Farey graph}\label{FaSec}

The structure of the Farey graph is as follows (see~\cite{HW}).
The set of vertices  consists of rational numbers $\Q$, completed by $\infty:=\frac{1}{0}$.
Two rationals,~$\frac{r}{s}$ and~$\frac{r'}{s'}$ (always written as irreducible fractions), 
are connected by an edge if and only if $rs'-r's=\pm1$.
Edges of the Farey graph are often represented
as (non-crossing) geodesics of the hyperbolic plane.

Although the Farey graph is much more complicated than the Pascal triangle,
in particular every vertex $\frac{r}{s}$ has infinitely many neighbors,
these two graphs have one property in common: each vertex has two ``parents''.
In the Farey graph these ``parents'' are characterized as follows.
Among the infinite set of neighbors of~$\frac{r}{s}$ there are exactly two, $\frac{r'}{s'}$ and $\frac{r''}{s''}$,
that are also connected to each other.
In other words, every rational~$\frac{r}{s}$ belongs to exactly one triangle
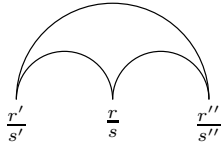
\begin{center}
\psscalebox{1.0 1.0} 
{
\psset{unit=0.8cm}
\begin{pspicture}(0,-1.315)(3.385,1.315)
\definecolor{colour0}{rgb}{1.0,0.0,0.2}
\psarc[linecolor=black, linewidth=0.02, dimen=outer](0.87,-0.685){0.8}{0.0}{180.0}
\psarc[linecolor=black, linewidth=0.02, dimen=outer](2.47,-0.685){0.8}{0.0}{180.0}
\psarc[linecolor=black, linewidth=0.02, dimen=outer](1.67,-0.685){1.6}{0.0}{180.0}
\rput(0.07,-1.085){$\frac{r'}{s'}$}
\rput(1.67,-1.085){$\frac{r}{s}$}
\rput(3.27,-1.085){$\frac{r''}{s''}$}
\end{pspicture}
}
\end{center}
such that $\frac{r'}{s'}<\frac{r}{s}<\frac{r''}{s''}$.
Furthermore, one has $\frac{r}{s}=\frac{r'+r''}{s'+s''}$.

Similarly to the case of $q$-binomials,
the edges of the weighted Farey graph are labeled by powers of~$q$ according to the following pattern
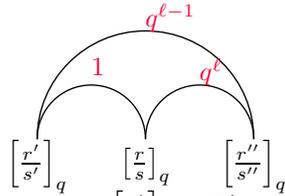
\begin{center}
\psscalebox{1.0 1.0} 
{
\psset{unit=0.9cm}
\begin{pspicture}(0,-1.315)(3.385,1.315)
\definecolor{colour0}{rgb}{1.0,0.0,0.2}
\psarc[linecolor=black, linewidth=0.02, dimen=outer](0.87,-0.685){0.8}{0.0}{180.0}
\psarc[linecolor=black, linewidth=0.02, dimen=outer](2.47,-0.685){0.8}{0.0}{180.0}
\psarc[linecolor=black, linewidth=0.02, dimen=outer](1.67,-0.685){1.6}{0.0}{180.0}
\rput[tl](0.87,0.515){\textcolor{colour0}{1}}
\rput[tl](2.47,0.515){\textcolor{colour0}{$q^\ell$}}
\rput[tl](1.67,1.315){\textcolor{colour0}{$q^{\ell-1}$}}
\rput(0.07,-1.085){$\left[\frac{r'}{s'}\right]_{q}$}
\rput(1.67,-1.085){$\left[\frac{r}{s}\right]_{q}$}
\rput(3.27,-1.085){$\left[\frac{r''}{s''}\right]_{q}$}
\end{pspicture}
}
\end{center}
The vertices are labeled by the following rule:
if $\left[\frac{r'}{s'}\right]_{q}=\frac{\Rc'}{\Sc'}$ and  $\left[\frac{r''}{s''}\right]_{q}=\frac{\Rc''}{\Sc''}$, then
\begin{equation}
\label{FSEq}
\left[\frac{r}{s}\right]_{q}:=\frac{\Rc'+q^{\ell}\Rc''}{\Sc'+q^{\ell}\Sc''}.
\end{equation}
Note that \eqref{FSEq} is analogous to \eqref{PascalEq}.

The weights $q^\ell$ and the $q$-rationals can be calculated recursively along the Farey graph; see Figure~\ref{wtFg}.
\begin{figure}[htbp]
\begin{center}
\psscalebox{1.0 1.0} 
{
\psset{unit=0.9cm}
\begin{pspicture}(0,-4.535)(13.757692,4.535)
\psdots[linecolor=black, dotsize=0.12](3.2788463,-2.665)
\psdots[linecolor=black, dotsize=0.12003673](3.2788463,-2.665)
\psdots[linecolor=black, dotsize=0.12](0.07884616,-2.665)
\psdots[linecolor=black, dotsize=0.12](13.678846,-2.665)
\psdots[linecolor=black, dotsize=0.12](2.478846,-2.665)
\psdots[linecolor=black, dotsize=0.12](1.6788461,-2.665)
\psdots[linecolor=black, dotsize=0.12](0.87884617,-2.665)
\psdots[linecolor=black, dotsize=0.12](0.07884616,-2.665)
\psdots[linecolor=black, dotsize=0.12](4.078846,-2.665)
\psdots[linecolor=black, dotsize=0.12](4.878846,-2.665)
\psdots[linecolor=black, dotsize=0.12](5.6788464,-2.665)
\psdots[linecolor=black, dotsize=0.12](6.478846,-2.665)
\psdots[linecolor=black, dotsize=0.12](7.2788463,-2.665)
\psdots[linecolor=black, dotsize=0.12](8.078846,-2.665)
\psdots[linecolor=black, dotsize=0.12](8.878846,-2.665)
\psdots[linecolor=black, dotsize=0.12](9.678846,-2.665)
\psdots[linecolor=black, dotsize=0.12](10.478847,-2.665)
\psdots[linecolor=black, dotsize=0.12](11.278846,-2.665)
\psdots[linecolor=black, dotsize=0.12](12.078846,-2.665)
\psdots[linecolor=black, dotsize=0.12](12.878846,-2.665)
\psarc[linecolor=black, linewidth=0.02, dimen=outer](6.878846,-2.665){6.8}{0.0}{180.0}
\psdots[linecolor=black, dotsize=0.12](13.678846,-2.665)
\psarc[linecolor=black, linewidth=0.02, dimen=outer](6.878846,-2.665){0.4}{0.0}{180.0}
\psarc[linecolor=black, linewidth=0.02, dimen=outer](7.6788464,-2.665){0.4}{0.0}{180.0}
\psarc[linecolor=black, linewidth=0.02, dimen=outer](8.478847,-2.665){0.4}{0.0}{180.0}
\psarc[linecolor=black, linewidth=0.02, dimen=outer](9.278846,-2.665){0.4}{0.0}{180.0}
\psarc[linecolor=black, linewidth=0.02, dimen=outer](10.078846,-2.665){0.4}{0.0}{180.0}
\psarc[linecolor=black, linewidth=0.02, dimen=outer](10.878846,-2.665){0.4}{0.0}{180.0}
\psarc[linecolor=black, linewidth=0.02, dimen=outer](11.678846,-2.665){0.4}{0.0}{180.0}
\psarc[linecolor=black, linewidth=0.02, dimen=outer](12.478847,-2.665){0.4}{0.0}{180.0}
\psarc[linecolor=black, linewidth=0.02, dimen=outer](13.278846,-2.665){0.4}{0.0}{180.0}
\psarc[linecolor=black, linewidth=0.02, dimen=outer](0.47884616,-2.665){0.4}{0.0}{180.0}
\psarc[linecolor=black, linewidth=0.02, dimen=outer](1.2788461,-2.665){0.4}{0.0}{180.0}
\psarc[linecolor=black, linewidth=0.02, dimen=outer](2.0788462,-2.665){0.4}{0.0}{180.0}
\psarc[linecolor=black, linewidth=0.02, dimen=outer](2.8788462,-2.665){0.4}{0.0}{180.0}
\psarc[linecolor=black, linewidth=0.02, dimen=outer](3.6788461,-2.665){0.4}{0.0}{180.0}
\psarc[linecolor=black, linewidth=0.02, dimen=outer](4.478846,-2.665){0.4}{0.0}{180.0}
\psarc[linecolor=black, linewidth=0.02, dimen=outer](5.2788463,-2.665){0.4}{0.0}{180.0}
\psarc[linecolor=black, linewidth=0.02, dimen=outer](6.078846,-2.665){0.4}{0.0}{180.0}
\psarc[linecolor=black, linewidth=0.02, dimen=outer](1.6788461,-2.665){0.8}{0.0}{180.0}
\psarc[linecolor=black, linewidth=0.02, dimen=outer](3.2788463,-2.665){0.8}{0.0}{180.0}
\psarc[linecolor=black, linewidth=0.02, dimen=outer](4.878846,-2.665){0.8}{0.0}{180.0}
\psarc[linecolor=black, linewidth=0.02, dimen=outer](6.478846,-2.665){0.8}{0.0}{180.0}
\psarc[linecolor=black, linewidth=0.02, dimen=outer](8.078846,-2.665){0.8}{0.0}{180.0}
\psarc[linecolor=black, linewidth=0.02, dimen=outer](9.678846,-2.665){0.8}{0.0}{180.0}
\psarc[linecolor=black, linewidth=0.02, dimen=outer](11.278846,-2.665){0.8}{0.0}{180.0}
\psarc[linecolor=black, linewidth=0.02, dimen=outer](12.878846,-2.665){0.8}{0.0}{180.0}
\psarc[linecolor=black, linewidth=0.02, dimen=outer](2.478846,-2.665){1.6}{0.0}{180.0}
\psarc[linecolor=black, linewidth=0.02, dimen=outer](4.878846,-2.665){0.8}{0.0}{180.0}
\psarc[linecolor=black, linewidth=0.02, dimen=outer](5.6788464,-2.665){1.6}{0.0}{180.0}
\psarc[linecolor=black, linewidth=0.02, dimen=outer](8.878846,-2.665){1.6}{0.0}{180.0}
\psarc[linecolor=black, linewidth=0.02, dimen=outer](12.078846,-2.665){1.6}{0.0}{180.0}
\psarc[linecolor=black, linewidth=0.02, dimen=outer](4.078846,-2.665){3.2}{0.0}{180.0}
\psarc[linecolor=black, linewidth=0.02, dimen=outer](10.478847,-2.665){3.2}{0.0}{180.0}
\psarc[linecolor=black, linewidth=0.02, dimen=outer](7.2788463,-2.665){6.4}{0.0}{180.0}
\rput[b](6.878846,3.735){\textcolor{red}{1}}
\rput[t](4.078846,0.935){\textcolor{red}{1}}
\rput[t](10.478847,0.935){\textcolor{red}{$q$}}
\rput[t](2.478846,-0.665){\textcolor{red}{1}}
\rput[t](5.6788464,-0.665){\textcolor{red}{$q$}}
\rput[t](8.878846,-0.665){\textcolor{red}{1}}
\rput[t](12.078846,-0.665){\textcolor{red}{$q^2$}}
\rput[t](12.878846,-1.465){\textcolor{red}{$q^3$}}
\rput[b](11.278846,-1.865){\textcolor{red}{1}}
\rput[b](8.078846,-1.865){\textcolor{red}{1}}
\rput[b](9.678846,-1.865){\textcolor{red}{$q$}}
\rput[b](1.6788461,-1.865){\textcolor{red}{1}}
\rput[b](3.2788463,-1.865){\textcolor{red}{$q$}}
\rput[t](6.478846,-1.465){\textcolor{red}{$q^2$}}
\rput[b](4.878846,-1.865){\textcolor{red}{1}}
\rput[t](0.47884616,-1.865){\textcolor{red}{1}}
\rput[t](6.078846,4.535){\textcolor{red}{$q^{-1}$}}
\rput(0.07884616,-3.065){$\frac01$}
\rput(0.87884617,-3.065){$\frac11$}
\rput(1.7,-3.065){$\left[\frac{5}{4}\right]_q$}
\rput(13.678846,-3.065){$\frac10$}
\rput(7.3,-3.065){$\left[\frac{2}{1}\right]_q$}
\rput(4.1,-3.065){$\left[\frac{3}{2}\right]_q$}
\rput(4.9,-3.065){$\left[\frac{8}{5}\right]_q$}
\rput(10.6,-3.065){$\left[\frac{3}{1}\right]_q$}
\rput(2.478846,-3.065){$\left[\frac{4}{3}\right]_q$}
\rput(3.3,-3.065){$\left[\frac{7}{5}\right]_q$}
\rput(5.7,-3.065){$\left[\frac{5}{3}\right]_q$}
\rput(6.5,-3.065){$\left[\frac{7}{3}\right]_q$}
\rput(8.178847,-3.065){$\left[\frac{5}{2}\right]_q$}
\rput(12.2,-3.065){$\left[\frac{4}{1}\right]_q$}
\rput(12.99,-3.065){$\left[\frac{5}{1}\right]_q$}
\rput[b](1.2788461,-2.265){\textcolor{red}{\footnotesize$1$}}
\rput[b](2.0788462,-2.265){\textcolor{red}{\footnotesize$q$}}
\rput[b](2.8788462,-2.265){\textcolor{red}{\footnotesize$1$}}
\rput[b](4.478846,-2.265){\textcolor{red}{\footnotesize$1$}}
\rput[b](6.078846,-2.265){\textcolor{red}{\footnotesize$1$}}
\rput[b](7.6788464,-2.265){\textcolor{red}{\footnotesize$1$}}
\rput[b](9.278846,-2.265){\textcolor{red}{\footnotesize$1$}}
\rput[b](10.878846,-2.265){\textcolor{red}{\footnotesize$1$}}
\rput[b](12.478847,-2.265){\textcolor{red}{\footnotesize$1$}}
\rput[br](3.6788461,-2.265){\textcolor{red}{\footnotesize$q^2$}}
\rput[br](5.2788463,-2.265){\textcolor{red}{\footnotesize$q$}}
\rput[br](6.878846,-2.265){\textcolor{red}{\footnotesize$q^3$}}
\rput[br](8.478847,-2.265){\textcolor{red}{\footnotesize$q$}}
\rput[br](10.078846,-2.265){\textcolor{red}{\footnotesize$q^2$}}
\rput[br](11.678846,-2.265){\textcolor{red}{\footnotesize$q$}}
\rput[br](13.278846,-2.265){\textcolor{red}{\footnotesize$q^4$}}
\rput(9.0,-3.065){$\left[\frac{8}{3}\right]_q$}
\rput(9.8,-3.065){$\left[\frac{11}{4}\right]_q$}
\rput(11.408846,-3.065){$\left[\frac{7}{2}\right]_q$}
\end{pspicture}
}
\caption{Upper part of the weighted Farey graph between~$\frac01$ and~$\frac10$}
\label{wtFg}
\end{center}
\end{figure}
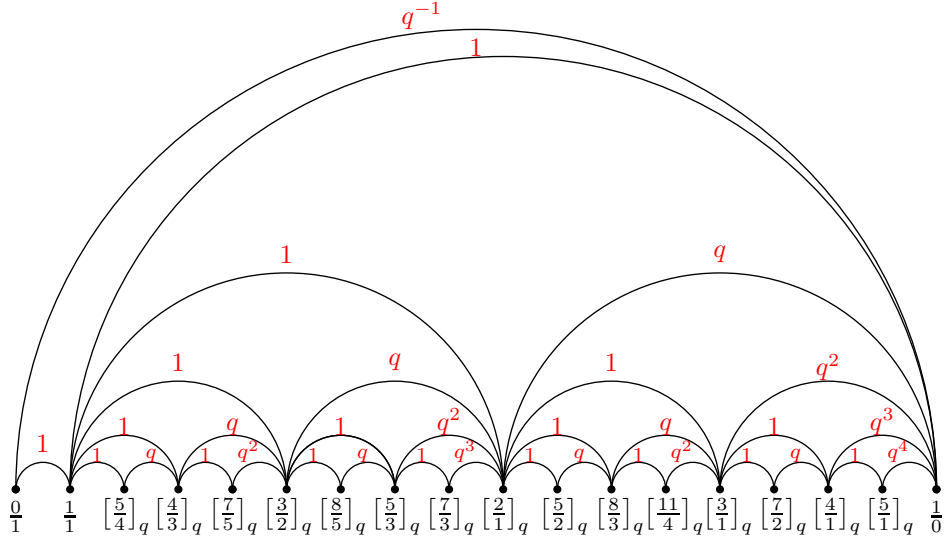

\noindent
For instance, $\left[\frac{7}{5}\right]_q=\frac{1+q+2q^2+2q^3+q^4}{1+q+2q^2+q^3}$.
We will be interested by
the corresponding Taylor series at $q=0$.
For instance, for the above example the series starts as follows
$$
\left[\frac{7}{5}\right]_q=
1+ q^3 - 2q^5 + q^6 + 3q^7 - 3q^8- 4q^9 + 7q^{10} + 4q^{11}- 14q^{12} \pm\cdots
$$

\subsection{$q$-deformed continued fractions}
We will now give another method of computing $q$-rationals.
It is relied on the continued fraction expansion.

Given a rational number $\frac{r}{s}>1$ , where $r$ and $s$ are positive integers assumed to be coprime.
It has a unique continued fraction expansion with even number of terms
\begin{equation}
\label{CFEq}
\frac{r}{s}
\quad=\quad
a_1 + \cfrac{1}{a_2 
          + \cfrac{1}{\ddots +\cfrac{1}{a_{2m}} } },
\end{equation}
with $a_i\geq1$, denoted by
$[a_1,\ldots,a_{2m}]$.
Note that the choice of even length removes the ambiguity
$[a_1,\ldots,a_{n},1]=[a_1,\ldots,a_{n}+1]$ and makes the expansion unique.

In order to calculate the $q$-deformation $\left[\frac{r}{s}\right]_q$,
one can use the following explicit formula.
Given a regular continued fraction $[a_{1}, \ldots, a_{2m}]$,
its $q$-deformation is given by
\begin{equation}
\label{qa}
[a_{1}, \ldots, a_{2m}]_{q}:=
[a_1]_{q} + \cfrac{q^{a_{1}}}{[a_2]_{q^{-1}} 
          + \cfrac{q^{-a_{2}}}{[a_{3}]_{q} 
          +\cfrac{q^{a_{3}}}{[a_{4}]_{q^{-1}}
          + \cfrac{q^{-a_{4}}}{
        \cfrac{\ddots}{[a_{2m-1}]_q+\cfrac{q^{a_{2m-1}}}{[a_{2m}]_{q^{-1}}}}}
          } }} 
\end{equation}
where we use the standard notation for the $q$-integers~$[a]_q=1+q+q^2+\cdots+q^{a-1}$.
Of course, one can get rid of negative exponents in~\eqref{qa}, but the formula becomes uglier.

\subsection{The matrix formulas}
We give another, equivalent, form to define $q$-rationals.
It uses $2\times2$ matrices and is well adapted for computer programming.

Let, as before, $\frac{r}{s}$ be a rational written in the form of a continued fraction expansion~\eqref{CFEq}.
Consider the $2\times2$ matrix with polynomial coefficients that was denoted by $\widetilde{M}^{+}_{q}(a_{1},\ldots, a_{2m})$
in~\cite{MGOqR}. 
It is defined by the formula
\begin{equation}
\label{DegEqOne}
\widetilde{M}^{+}_{q}(a_{1},\ldots, a_{2m}):=
\begin{pmatrix}
[a_{1}]_{q}&q^{a_{1}}\\[6pt]
1&0
\end{pmatrix}
\begin{pmatrix}
q[a_{2}]_{q}& 1\\[6pt]
q^{a_{2}}&0
\end{pmatrix}
\cdots
\begin{pmatrix}
[a_{2m-1}]_{q}&q^{a_{2m-1}}\\[6pt]
1&0
\end{pmatrix}
\begin{pmatrix}
q[a_{2m}]_{q}&1\\[6pt]
q^{a_{2m}}&0
\end{pmatrix}.
\end{equation}
This matrix is a $q$-analogue of the usual ``matrix of convergents'' of a continued fraction (see~\cite{Never,MGO}).

It is easy to prove that this matrix contains the numerator and denominator of~$\left[\frac{r}{s}\right]_q$ (up to multiplication by~$q$)
in the first column:
\begin{equation}
\label{qRegMat}
\widetilde{M}^{+}_{q}(a_{1},\ldots, a_{2m})=
\begin{pmatrix}
q\Rc&\Rc'_{2m-1}\\[6pt]
q\Sc&\Sc'_{2m-1}
\end{pmatrix},
\end{equation}
where $\frac{\Rc(q)}{\Sc(q)}=\left[\frac{r}{s}\right]_q, $ and where 
$\frac{\Rc'_{2m-1}(q)}{\Sc'_{2m-1}(q)}=[a_{1}, \ldots,a_{2m-1}]_{q}$ is the previous convergent.

\section{The stabilization phenomenon}\label{ProoSec}

In this section we 
first collect some basic properties that follow from the matrix presentation.
We prove Proposition~\ref{TechLem} and Theorem~\ref{ConvThm}.
Finally, in a remark, we discuss the stabilization phenomenon in the case when $x$ is rational.

\subsection{Some simple properties of $q$-rationals}
The following statements are immediate corollaries of~\eqref{DegEqOne} and~\eqref{qRegMat}.

Let $\frac{r}{s}\geq1$ be a rational number.
Then $\left[\frac{r}{s}\right]_q=\frac{\Rc(q)}{\Sc(q)}$, where 
$\Rc(q)$ and $\Sc(q)$ 
are polynomials with positive coefficients whose highest and lowest coefficients are equal to~$1$.
The degrees of the polynomials~${\Rc}$ and~${\Sc}$ are as follows
$$
\begin{array}{rcl}
\deg(\Rc)&=&a_{1}+\ldots +a_{2m}-1,\\[4pt]
\deg(\Sc)&=&a_{2}+\ldots +a_{2m}-1.
\end{array}
$$
The unimodality conjecture of~\cite{MGOqR} states that the coefficients of~${\Rc}$ and ${\Sc}$ first grow and then decrease monotonically.

Let us mention that the most important properties of $q$-rationals 
are the total positivity and the combinatorial interpretation of the coefficients of~$\Rc$ and~$\Sc$.

\subsection{Proof of Proposition~\ref{TechLem}}\label{PrPSec}

Let~$x\geq1$ be a real number, and let $x_{n-1}$ and~$x_n$ be two consecutive convergents of its continued fraction.
Let 
$$
\left[x_{n-1}\right]_q=\frac{\Rc_{n-1}}{\Sc_{n-1}}
\qquad\hfill{and}\qquad
\left[x_n\right]_q=\frac{\Rc_n}{\Sc_n},
$$
One then has tautologically
$$
\frac{\Rc_n}{\Sc_n}-
\frac{\Rc_{n-1}}{\Sc_{n-1}}=
\frac{\Rc_n\Sc_{n-1}-\Sc_n\Rc_{n-1}}{\Sc_n\Sc_{n-1}}.
$$
The polynomial in the numerator of the right-hand-side is a power of~$q$:
\begin{equation}
\label{PowerLem}
\Rc_n\Sc_{n-1}-\Sc_n\Rc_{n-1}=q^{a_1+\cdots+a_n-1}.
\end{equation}
Indeed, to prove this,
it suffices to compare the determinant of the matrix $\widetilde{M}^{+}_{q}(a_{1},\ldots, a_{2m})$
written in the forms~\eqref{DegEqOne} and~\eqref{qRegMat}.

Both polynomials,
$\Sc_n$ and $\Sc_{n-1}$, start with zero-order term~$1$, and so is the series $1/(\Sc_n\Sc_{n-1})$.
It follows now from~\eqref{PowerLem}, that the series $\left[x_n\right]_q-\left[x_{n-1}\right]_q$ is of the form 
$$
\left[x_n\right]_q-\left[x_{n-1}\right]_q=q^{a_1+\cdots+a_n-1}+O(q^{a_1+\cdots+a_n}).
$$
Hence, Proposition~\ref{TechLem}.

\subsection{Proof of Theorem~\ref{ConvThm}}

Let now $(y_n)_{n\geq1}$ be an arbitrary sequence of rationals converging to~$x$.
Then, for every fixed~$m$, there exists~$N$ such that
$y_n\in[x_{m-1},x_m]$ (or $y_n\in[x_{m},x_{m-1}])$ for every~$n\geq{}N$.
Here, as above, $(x_n)_{n\geq1}$ is the sequence of convergents of the continued fraction of~$x$.

By Proposition~\ref{TechLem}, the first $a=a_1+\cdots+a_{m}-1$ terms 
of the Taylor series of~$\left[x_{m-1}\right]_q$ and~$\left[x_m\right]_q$ coincide.
It turns out that the same is true for every rational between $x_{m-1}$ and~$x_m$.

\begin{lem}
\label{SopLem}
For every rational~$\frac{r}{s}$ such that, $x_{m-1}<\frac{r}{s}<x_m$, 
the first $a=a_1+\cdots+a_{m}-1$ terms  of the Taylor series of~$\left[\frac{r}{s}\right]_q$ coincides with those of~$\left[x_{m-1}\right]_q$ and~$\left[x_m\right]_q$.
\end{lem}

\begin{proof}
Let~$\left[x_{m-1}\right]_q=\frac{\Rc_{m-1}}{\Sc_{m-1}}$ and~$\left[x_{m}\right]_q=\frac{\Rc_{m}}{\Sc_{m}}$.
Recall that  $x_{m-1}$ and~$x_m$ are joined by an edge in the Farey graph.

Suppose first that~$\frac{r}{s}$ is also joint to $x_{m-1}$ and~$x_m$, so that we have a triangle
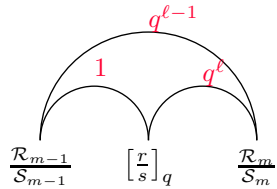
\begin{center}
\psscalebox{1.0 1.0} 
{
\psset{unit=0.9cm}
\begin{pspicture}(0,-1.315)(3.385,1.315)
\definecolor{colour0}{rgb}{1.0,0.0,0.2}
\psarc[linecolor=black, linewidth=0.02, dimen=outer](0.87,-0.685){0.8}{0.0}{180.0}
\psarc[linecolor=black, linewidth=0.02, dimen=outer](2.47,-0.685){0.8}{0.0}{180.0}
\psarc[linecolor=black, linewidth=0.02, dimen=outer](1.67,-0.685){1.6}{0.0}{180.0}
\rput[tl](0.87,0.515){\textcolor{colour0}{1}}
\rput[tl](2.47,0.515){\textcolor{colour0}{$q^\ell$}}
\rput[tl](1.67,1.315){\textcolor{colour0}{$q^{\ell-1}$}}
\rput(0.07,-1.085){$\frac{\Rc_{m-1}}{\Sc_{m-1}}$}
\rput(1.67,-1.085){$\left[\frac{r}{s}\right]_{q}$}
\rput(3.27,-1.085){$\frac{\Rc_m}{\Sc_m}$}
\end{pspicture}
}
\end{center}
Set $\left[\frac{r}{s}\right]_q=\frac{\Rc}{\Sc}$, then, by definition of $q$-rationals,
$$
\frac{\Rc}{\Sc}=\frac{\Rc_{m-1}+q^\ell\Rc_m}{\Sc_{m-1}+q^\ell\Sc_m}
$$
By~\eqref{PowerLem}, we have $\Rc_{m}\Sc_{m-1}-\Sc_{m}\Rc_{m-1}=q^a$.
Therefore, 

\begin{equation}\label{eqdisym}
\begin{array}{rclcl}
\Rc\Sc_{m-1}-\Sc\Rc_{m-1}&=&\Rc_{m-1}\Sc_{m-1}+q^\ell\Rc_m\Sc_{m-1}-\Rc_{m-1}\Sc_{m-1}-q^\ell\Rc_{m-1}\Sc_m&=&q^{a+\ell}\\[4pt]
\Rc\Sc_m-\Sc\Rc_m&=&\Rc_{m-1}\Sc_m+q^\ell\Rc_m\Sc_m-\Rc_m\Sc_{m-1}-q^\ell\Rc_m\Sc_m&=&-q^{a}.
\end{array}
\end{equation}

Using the same argument as in the proof of Proposition~\ref{TechLem}
we deduce the statement of the lemma in the case where the three points form a triangle.

The general case of the lemma can then be proved inductively.
Indeed, every rational~$\frac{r}{s}$ such that, $x_{m-1}<\frac{r}{s}<x_m$ can be joined to $x_{m-1}$ and~$x_m$
by a sequence of triangles.
To see this, draw a vertical line
in the Poincar\'e half-plane through~$\frac{r}{s}$ and collect all the triangles of the Farey tessellation
between $x_{m-1}$ and~$x_m$ crossed in their interior by this line.

Hence, the lemma.
\end{proof}

Theorem~\ref{ConvThm} follows from Lemma~\ref{SopLem} and Proposition~\ref{TechLem}.

\begin{rem}
We also investigated the stabilization phenomenon in the case when $x$ is rational.
If $x=\frac{r}{s}$ and $\left(\frac{r_n}{s_n}\right)_{n\geq1}$ is a sequence converging to $x$ it is natural to ask whether 
$[x]_q$ defined by \eqref{TSEqBis} is equal to the $q$-rational $\left[\frac{r}{s}\right]_q$. 
The answer is surprising: when the sequence $\left(\frac{r_n}{s_n}\right)$ approaches $\frac{r}{s}$ from the right the stabilized power series defined by \eqref{TSEqBis} is equal to the $q$-rational 
$\left[\frac{r}{s}\right]_q$, when the sequence approaches the rational from the left this is no longer true. This is due to the asymmetry of the relations \eqref{eqdisym}.

Indeed, if $\frac{r_n}{s_n}>\frac{r}{s}$ for all $n>\!\!>0$, one can use the same arguments as in the proof of Theorem~\ref{ConvThm} by replacing the sequence of convergents $x_n$ by
the sequence of right neighbors $\frac{nr+r''}{ns+s''}$, where $\frac{r''}{s''}$ is the right parent introduced in Section~\ref{FaSec}.
Considering Taylor expansions, the right neighbors $\left[\frac{nr+r''}{ns+s''}\right]_q$ have more and more first terms identical to those of $\left[\frac{r}{s}\right]_q$ as $n$ grows.

If $\frac{r_n}{s_n}<\frac{r}{s}$ for all $n>\!\!>0$ the above arguments no longer apply.
However, with experimental computations, we observed some
stabilization phenomenons for the sequences~$\left[\frac{r_n}{s_n}\right]_q$, 
but the stabilized power series is different from $\left[\frac{r}{s}\right]_q$.  
For instance, testing several sequences of rational approaching the integer $2$ with smaller values we always obtained a stabilization to the series $1+q^2$ which is not $[2]_{q}=1+q$.

\end{rem}

\section{$q$-deformations of quadratic irrationals}\label{ExSec}

In this section we discuss several examples of quadratic irrational numbers.
We start from the simplest possible case of the golden ratio and identify the coefficients of the
Taylor series as the remarkable and thoroughly studied sequence of generalized Catalan numbers.
We dwell on this first example with more details to better explain the stabilization phenomenon.
We then calculate the $q$-deformation of the number $1+\sqrt{2}$, which is usually called the ``silver ratio''.
Finally, we consider several examples of square roots of small positive integers
and calculate the corresponding functional equations.

\subsection{The golden ratio and generalized Catalan numbers}\label{GRSec}

The simplest example of an infinite continued fraction is the expansion of the golden ratio:
$$
\varphi=\frac{1+\sqrt{5}}{2}=
\left[1, 1, 1, 1, 1, \ldots\right].
$$
The convergents are ratios of consecutive Fibonacci numbers:
$\varphi_n=F_{n+1}/F_n$.

According to~\eqref{qa}, the $q$-deformation of~$\varphi$ is given by the $2$-periodic infinite
continued fraction
\begin{equation}
\label{RRAlterEq}
\left[\varphi\right]_q=
1 + \cfrac{q^{2}}{q
          + \cfrac{1}{1 
          +\cfrac{q^{2}}{q
          + \cfrac{1}{\ddots
       }}}} 
\end{equation}

\begin{rem}
Let us mention that there exists a celebrated $q$-deformation of the golden ratio, called the Rogers-Ramanujan continued fraction.
It is aperiodic and has a great number of beautiful and sophisticated properties (for a survey, see~\cite{RR}).
Unlike the Rogers-Ramanujan continued fraction, we do not know if~\eqref{RRAlterEq} is a quotient of two $q$-series with positive coefficients.
\end{rem}

Consider the $q$-deformations of the convergents $\left[\varphi_n\right]_q$.
We proved in~\cite{MGOqR} that the coefficients of the polynomials of
$\left[\varphi_n\right]_q$ are the numbers appearing in the remarkable ``Fibonacci lattice''
(see A123245 of OEIS~\cite{OEIS} and its mirror A079487).
More precisely, A123245 appears in the numerator and A079487 in the denominator.
For instance,
$$
\begin{array}{rcl}
\left[\varphi_6\right]_q&=&\displaystyle\frac{1+2q+3q^2+3q^3+3q^4+q^5}{1+2q+2q^2+2q^3+q^4},
\\[12pt]
\left[\varphi_8\right]_q&=&\displaystyle\frac{1+3q+5q^2+7q^3+7q^4+6q^5+4q^6+q^7}{1+3q+4q^2+5q^3+4q^4+3q^5+q^6},
\\[12pt]
\left[\varphi_9\right]_q&=&\displaystyle\frac{1+4q+7q^2+10q^3+11q^4+10q^5+7q^6+4q^7+q^8}{1+4q+6q^2+7q^3+7q^4+5q^5+3q^6+q^7}.
\end{array}
$$
We see that the coefficients at every fixed power of~$q$ of the numerator and the denominator grow.
To illustrate the stabilization phenomenon, we give the corresponding Taylor series:
$$
\begin{array}{rcl}
\left[\varphi_6\right]_q&=&
1 + q^2 - q^3 + 2 q^4 - 3 q^5 + 3 q^6 - 3 q^7 + 4 q^8 - 5 q^9 + 5 q^{10} - 5 q^{11} + 6 q^{12}\cdots\\[4pt]
\left[\varphi_8\right]_q&=&
1 + q^2 - q^3 + 2 q^4 - 4 q^5 + 8 q^6 - 16 q^7 + 30 q^8 - 55 q^9 + 103 q^{10} - 195 q^{11} + 368 q^{12}\cdots\\[4pt]
\left[\varphi_9\right]_q&=&
1 + q^2 - q^3 + 2 q^4 - 4 q^5 + 8 q^6 - 17 q^7 + 37 q^8 - 82 q^9 + 184 q^{10} - 414 q^{11} + 932 q^{12}\cdots
\end{array}
$$
The series $\left[\varphi_9\right]$ approximates $\left[\varphi\right]_q$ correctly up to the $9$th term, while $\left[\varphi_6\right]$ only up to $q^4$.

The full series~\eqref{RRAlterEq}
starts as follows:
$$
\begin{array}{rcl}
\left[\varphi\right]_q&=&
1 + q^2 - q^3 + 2 q^4 - 4 q^5 + 8 q^6 - 17 q^7 + 37 q^8 - 82 q^9 + 185 q^{10} - 423 q^{11} + 978 q^{12}-2283q^{13}\\[4pt]
&&+ 5373q^{14}-12735q^{15}+30372q^{16}-72832q^{17}+175502q^{18}-424748q^{19}+1032004q^{20} \cdots
\end{array}
$$
Note that one needs the $n$th convergent $\left[\varphi_n\right]_q$ to calculate~$\left[\varphi\right]_q$ with accuracy up to~$q^n$.

Fix the notation
$$
\left[\varphi\right]_q=:\sum_{k\geq0}\phi_kq^k.
$$ 
We were able to identify the coefficients $\phi_k$ appearing in this series 
as the so-called Generalized Catalan numbers
(see sequence A004148 of OEIS), but with alternating signs.

\begin{prop}
\label{GoldConj}
One has $\phi_k=(-1)^ka_{k-1}$, for $k\geq2$, where $a_k$ are the Generalized Catalan numbers; see~{\rm A004148} of~\cite{OEIS}.
\end{prop}

\begin{proof}
Let us prove
that the series $\left[\varphi\right]_q$ satisfies the following functional equation:
\begin{equation}
\label{GREq}
q\left[\varphi\right]^2_q-
\left(q^2+q-1 \right)\left[\varphi\right]_q -1 =0,
\end{equation}
which is a $q$-analogue of $\varphi^2=\varphi+1$.
In fact,~\eqref{GREq} is immediate consequence of~\eqref{RRAlterEq} which can also be written\footnote{This observation is due to Doron Zeilberger.}
$$
\left[\varphi\right]_q=
1 + \cfrac{q^{2}}{q
          + \cfrac{1}{\left[\varphi\right]_q
          }} .
$$

Proposition~\ref{GoldConj} then follows from the known results about the generating function of the Generalized Catalan numbers.
Indeed, this generating function satisfies an equation equivalent to~\eqref{GREq}.
(see M. Somos' contribution to A004148).
\end{proof}

Solving~\eqref{GREq}, one obtains
$$
\left[\varphi\right]_q=
\frac{q^2+q-1+\sqrt{q^4+2q^3-q^2+2q+1}}{2q}.
$$
Obviously, at $q=1$ one recovers the golden ratio~$\varphi$.

\begin{rem}
Let us also mention that
the coefficients of $\left[\varphi\right]_q$ satisfy, for all~$n\geq3$, the following linear recurrence
$$
(k+1)\phi_{k} +(2k-1)\phi_{k-1} +(2-k)\phi_{k-2} +(2k-7)\phi_{k-3} +(k-5)\phi_{k-4}=0.
$$
In the context of the generalized Catalan numbers, this recurrence was conjectured by R.J. Mathar in 2011, and recently proved; see~\cite{EYZ} (based on~\cite{Zei}).
\end{rem}

\subsection{The $q$-deformed silver ratio}

The number $1+\sqrt{2}=\left[2,2,2,2,\ldots\right]$ is often called the silver ratio.
It is denoted by~$\delta_S$,
and its convergents are given by the quotient of consecutive Pell numbers.
This is probably the next simplest example of infinite continued fraction after the golden ratio.

Formula~\eqref{qa} implies the following $q$-deformation
\begin{equation}
\label{RRSEq}
\left[\delta_S\right]_q=
1 +q+ \cfrac{q^{4}}{q+q^2
          + \cfrac{1}{1 +q
          +\cfrac{q^{4}}{q+q^2
          + \cfrac{1}{\ddots
       }}}} 
\end{equation}

The stabilization process goes twice faster than for the golden ratio:
one needs the $n$th convergent to calculate~$\left[\delta_S\right]_q$ with accuracy up to~$q^{2n}$.
The series~$\left[\delta_S\right]_q$ starts as follows
$$
\begin{array}{rcl}
\left[\delta_S\right]_q&=&
1+q+q^4-2q^6+q^7+4q^8-5q^9-7q^{10}+ 18q^{11}+ 7q^{12}-55q^{13}+ 18q^{14}\\[4pt]
&&+ 146q^{15}- 155q^{16} - 322q^{17}+692q^{18}+ 476q^{19}- 2446q^{20}+ 307q^{21}\\[4pt]
&&+ 7322q^{22}- 6276q^{23}- 18277q^{24}+ 33061q^{25}+ 33376q^{26}- 129238q^{27}- 10899q^{28}\cdots
\end{array}
$$
We see that the coefficients grow much slower than those of~$\left[\varphi_n\right]_q$.
This sequence of coefficients is not in~OEIS.

\begin{prop}
\label{MyFla}
The series $\left[\delta_S\right]_q$ satisfies the following functional equation:
\begin{equation}
\label{SREq}
q\left[\delta_S\right]^2_q-
\left(q^3+2q-1 \right)\left[\delta_S\right]_q -1 =0.
\end{equation}
\end{prop}

\begin{proof}
Fromula~\eqref{RRSEq} rewritten in the form
$$
\left[\delta_S\right]_q=
1 +q+ \cfrac{q^{4}}{q+q^2
          + \cfrac{1}{\left[\delta_S\right]_q
          }} 
$$
readily implies~\eqref{SREq}.
\end{proof}

Equation~\eqref{SREq} is a $q$-analogue of~$\delta_S^2=2\delta_S+1$, appearance of~$q^3$ in this formula is somewhat surprising.
Solving~\eqref{SREq}, one obtains
$$
\left[\delta_S\right]_q=
\frac{q^3+2q-1+\sqrt{q^6+4q^4-2q^3+4q^2+1}}{2q}.
$$

\subsection{The $q$-square roots of $2,3,5$ and $7$}
We calculate the $q$-analogues of the simplest square roots.
Recall that $\sqrt{2}=[1,\overline{2}],\, \sqrt{3}=[1,\overline{1,2}],\, \sqrt{5}=[2,\overline{4}],\, \sqrt{7}=[2,\overline{1,1,1,4}]$.
The series start as follows
$$
\begin{array}{rcl}
\left[\sqrt{2}\right]_q&=&
1+ q^3 - 2q^5  + q^6+ 4q^7- 5q^8- 7q^9+18q^{10} + 7q^{11}- 55q^{12}+ 18q^{13}\\[4pt]
&&+ 146q^{14} - 155q^{15}- 322q^{16}+ 692q^{17}+ 476q^{18}
 - 2446q^{19}+ 307q^{20}\\[4pt]
&&+ 7322q^{21}- 6276q^{22}- 18277q^{23}+ 33061q^{24}+ 33376q^{25}\cdots\\[8pt]
\left[\sqrt{3}\right]_q&=&
1+q^2-q^4+2q^5-2q^6-q^7+ 7q^8- 12q^9 + 7q^{10} + 18q^{11}- 59q^{12}+ 78q^{13}\\[4pt]
&&- q^{14}- 228q^{15}+ 514q^{16}- 469q^{17}- 506q^{18}+ 2591q^{19}- 4338q^{20}\\[4pt]
&&+ 1837q^{21}+ 9405q^{22} - 27430q^{23}+ 33390q^{24}+10329q^{25}\cdots\\[8pt]
\left[\sqrt{5}\right]_q&=&
1+q+q^6-q^8- q^9- q^{10}+ 3q^{11}+ 4q^{12}- q^{13} -6q^{14}- 11q^{15}+ 2q^{16}\\[4pt]
&&+ 25q^{17}+ 22q^{18}- 10q^{19}- 70q^{20}- 71q^{21}+ 67q^{22}+ 208q^{23}+ 168q^{24}- 222q^{25}\cdots\\[8pt]
\left[\sqrt{7}\right]_q&=&
1+q+q^3-q^4+2q^5-3q^6+4q^7-6q^8+8q^9-9q^{10}+9q^{11}-5q^{12}-9q^{13}\\[4pt]
&&+ 40q^{14}- 101q^{15}+ 215q^{16}- 411q^{17}+ 724q^{18}- 1195q^{19}+ 1845q^{20}\\[4pt]
&&- 2623q^{21}+ 3324q^{22}- 3412q^{23}+ 1696q^{24}+ 4157q^{25}\cdots
 \end{array}
$$
Note that the coefficients of~$\left[\sqrt{2}\right]_q$ are those of the silver ratio, but the power of~$q$ is shifted by~$1$;
in full accordance with~\eqref{TransEq}.

The following formulas can be proved in a similar way as~\eqref{GREq} and~\eqref{SREq},
The calculations are quite long but straightforward, so we omit the details.

\begin{prop}
\label{MySFla}
The series $\left[\sqrt{2}\right]_q,\left[\sqrt{3}\right]_q,\left[\sqrt{5}\right]_q$ and~$\left[\sqrt{7}\right]_q$ satisfy the following functional equations:
\begin{eqnarray}
\label{CREq}
q^2\left[\sqrt{2}\right]^2_q-
\left(q^3-1 \right)\left[\sqrt{2}\right]_q&=&q^2 + 1,\\
\label{CRTEq}
q^2\left[\sqrt{3}\right]^2_q-
\left(q^3+q^2-q-1 \right)\left[\sqrt{3}\right]_q&=&q^2+q + 1,\\
\label{RFiveEq}
q^3\left[\sqrt{5}\right]_q^2-(q^5+q^3-q^2-1)\left[\sqrt{5}\right]_q&=&q^4+q^3+q^2+q+1,\\
\label{RFSevenEq}
q^3\left[\sqrt{7}\right]_q^2-(q^5+q^4-q-1)\left[\sqrt{7}\right]_q&=&q^4+2q^3+q^2+2q+1.
\end{eqnarray}
\end{prop}

We obtain, consequently, the following expressions for the quantized square roots.
$$
\begin{array}{rcl}
\left[\sqrt{2}\right]_q&=&
\displaystyle
\frac{q^3-1+\sqrt{q^6+4q^4-2q^3+4q^2+1}}{2q^2},
\\[8pt]
\left[\sqrt{3}\right]_q&=&
\displaystyle
\frac{q^3+q^2-q-1+\sqrt{q^6 + 2q^5 + 3q^4 + 3q^2 + 2q + 1}}{2q^2},
\\[8pt]
\left[\sqrt{5}\right]_q&=&
\displaystyle
\frac{q^5+q^3-q^2-1+\sqrt{q^{10}+2q^8+2q^7 + 5q^6 + 5q^4 + 2q^3 + 2q^2+ 1}}{2q^3},
\\[8pt]
\left[\sqrt{7}\right]_q&=&
\displaystyle
\frac{q^5+q^4-q-1+\sqrt{q^{10}+2q^9+q^8+4q^7 + 6q^6 + 6q^4 + 4q^3 + q^2+2q +1}}{2q^3}.
\end{array}
$$
Note that $\left[\sqrt{5}\right]_q$ looks quite different from the golden ratio.
This is an example of a highly non-trivial action of homothety $x\to{}x/2$.

We wonder if some relations similar to~\eqref{GREq},~\eqref{SREq},~\eqref{CREq},~\eqref{CRTEq},~\eqref{RFiveEq} and~\eqref{RFSevenEq}
hold for $q$-deformations of arbitrary quadratic irrationals.

\section{$q$-deformations of $e$ and $\pi$}\label{ExSecBis}

In this section we write down the first terms of the $q$-deformations of two notable examples 
of transcendental irrational numbers,~$e$ and~$\pi$.
We calculated several hundreds of terms to convince ourselves that
the coefficients of the corresponding series do not correspond to any sequence of OEIS.
However, one can make some surprising observations.

\subsection{Computing $\left[e\right]_q$}

The  continued fraction expansion of the Euler constant is given by the following famous regular pattern
$e=\left[2, 1, 2, 1, 1, 4, 1, 1, 6, 1, 1, 8, 1, 1, 10,\ldots\right]$ (see sequence A003417 in the OEIS).
To calculate the first $40$ terms in the series $\left[e\right]_q$, one needs to take the $15$th convergent
$$
e_{15}=\left[2, 1, 2, 1, 1, 4, 1, 1, 6, 1, 1, 8, 1, 1, 10\right]=517656/190435.
$$
The series $\left[e\right]_q$ starts as follows:
$$
\begin{array}{rcl}
\left[e\right]_q&=&
1+q+q^3-q^5+2q^6-3q^7+3q^8-q^9\\[4pt]
&&-3q^{10}+9q^{11}-17q^{12}+25q^{13}-29q^{14}+23q^{15}+2q^{16}\\[4pt]
&&-54q^{17}+ 134q^{18}- 232q^{19}+ 320q^{20} - 347q^{21}+ 243q^{22}+ 71q^{23}\\[4pt]
&&- 660q^{24}+1531q^{25}- 2575q^{26}+ 3504q^{27}- 3804q^{28}+ 2747q^{29}+ 488q^{30}\\[4pt]
&&- 6537q^{31}+ 15395q^{32}- 25819q^{33}+ 34716q^{34}- 36780q^{35}+ 24771q^{36}+ 9096q^{37}\\[4pt]
&& - 70197q^{38}+ 156811q^{39}\cdots
\end{array}
$$
We observe that the coefficients of~$q^{2+7k}$,
where $k\geq0$, turn out to be smaller than those of their neighbors.
The signs of the coefficients also obey a certain $7$-periodic pattern: indeed, the double plus,
i.e., ``$+,+$'' signs, appears with period~$7$.
We do not know any reason for such a ``$7$-periodicity'' related to Euler's number.

\subsection{The quantum $\pi$}

The continued fraction expansion of~$\pi$ (cf. sequence A001203 in the OEIS) starts as follows: 
$\pi=\left[3, 7, 15, 1, 292, 1, 1, 1, 2, 1, 3, 1, 14, 2, 1, 1, 2, 2, 2, 2, 1, 84,\ldots\right]$.
The rules governing this sequence are unknown.

The $4$th convergent $\pi_5=\left[3, 7, 15, 1\right]=355/113$ gives $24$ first terms of~$\left[\pi\right]_q$,
and already the $5$th convergent $\pi_5=\left[3, 7, 15, 1, 292\right]=103993/33102$
allows us to calculate~$\left[\pi\right]_q$ up to degree~${317}$.
We calculated up to
$\left[\pi_{38}\right]_q=11895062545096656711950/3786316004878788190109$
that approximates~$\pi$ up to~$43$ digits, and
that gives the first $603$ terms of the series $\left[\pi\right]_q$.
The first $79$ terms of $\left[\pi\right]_q$ are:
$$
\begin{array}{rcl}
\left[\pi\right]_q&=&
1+q+q^2+q^{10}-q^{12}-q^{13}+q^{15}+q^{16} \\[4pt]
&&- q^{20}-2q^{21}- q^{22}+2q^{23}+4q^{24}+q^{25}\\[4pt]
&&-4q^{27}-4q^{28}-2q^{29}+q^{30}+5q^{31}+8q^{32}+3q^{33}\\[4pt]
&&-3q^{34}-10q^{35}-12q^{36}-5q^{37}+8q^{38}+19q^{39}+20q^{40}+2q^{41}\\[4pt]
&&-18q^{42}-32q^{43}-25q^{44}+31q^{46}+51q^{47}+45q^{48}\\[4pt]
&&-7q^{49}-65q^{50}- 94q^{51}- 57q^{52}+ 35q^{53}+122q^{54}+ 140q^{55} + 72q^{56}\\[4pt]
&&- 76q^{57}- 209q^{58}- 234q^{59}- 90q^{60}
+ 171q^{61}+ 383q^{62}+ 363q^{63}+ 76q^{64}\\[4pt]
&&- 364q^{65}- 650q^{66} - 545q^{67}- 6q^{68}+ 702q^{69}+ 1101q^{70}+ 790q^{71}
\\[4pt]
&&- 180q^{72}- 1329q^{73}- 1824q^{74} - 1113q^{75}+ 642q^{76}+ 2454q^{77}+ 2982q^{78}
+ 1415q^{79} \cdots
\end{array}
$$
The coefficients of this series grow very slowly in contrast with the other examples we considered so far.
Asymptotic of the ratio of two consecutive coefficients seems to be close to~$1$,
that would imply that the radius of convergence of the series is equal to~$1$.
A curious observation is that, for unknown reasons, the coefficient of~$q^{45}$ vanishes.
One also observes oscillation of the sequence of coefficients, and the unimodality property
of every (short) subsequence of coefficients with constant sign.
We were unable to find any pattern or to conjecture any functional equation for this series.

\section{Translations}\label{TransSec}

In this section we consider the properties of $q$-deformed reals
under translations of the argument.
The action of the translation group~$\Z$ is defined by the operator
$T:x\to{}x+1$ and its inverse, $T^{-1}:x\to{}x-1$.
We first study~$T$ which is simpler, and then~$T^{-1}$ brings us to a new ground.
In particular, we extend the notion of $q$-deformation to negative real numbers.

\subsection{Right translations}
Consider first the action of~$T$.

\begin{prop}
\label{TroP}
If $x\leq1$, then $\left[x+1\right]_q=q\left[x\right]_q+1$.
\end{prop}

\begin{proof}
The statement of the proposition is obvious when $x$ is a $q$-integer.
Hence by the explicit formula~\eqref{qa} the statement is also clear for a rational~$x$.
 The irrational case will then follow from the stabilization phenomenon.
\end{proof}

\subsection{The coefficient ``gap''}

\begin{prop}
\label{GaP}
The first order term of series $\left[x\right]_q$ vanishes:
$$
\left[x\right]_q=1+\varkappa_2q^2+\varkappa_3q^3+\cdots
$$
if and only if $1\leq{}x\leq2$.
\end{prop}

\begin{proof}
As in the previous proof, it suffices to prove the statement for a rational~$x$.
Let $\frac{r}{s}=\left[a_1,a_2,\ldots,a_{2m}\right]$ be a rational written in a form of a continued fraction.
It is an easy computation that the polynomials $\Rc=1+r_1q+\cdots$ and $\Sc=1+s_1q+\cdots$
of the $q$-deformed rational $\left[\frac{r}{s}\right]_q=\frac{\Rc}{\Sc}$ have identical first-order coefficient:
$r_1=s_1$ if and only if $a_1=1$.

The result then follows from the formula
$$
\frac{\Rc}{\Sc}=\frac{1+r_1q+\cdots}{1+s_1q+\cdots}=(1+r_1q+\cdots)(1-s_1q+\cdots).
$$
\end{proof}

Propositions~\ref{TroP} and~\ref{GaP} together imply Theorem~\ref{GapThm}.

\subsection{Left translations and $q$-deformed negative numbers}
The heuristic definition 
$$
\left[x-1\right]_q:=\frac{\left[x\right]_q-1}{q},
$$
that we adopt now,
leads to a classes of $q$-deformed reals that we did not consider so far.
For~$x\geq2$, the action of~$T^{-1}$ is a tautological inverse of Proposition~\ref{TroP}.
The situation changes for $1\leq{}x\leq2$ because of the ``first order gap'' of Proposition~\ref{GaP}.

More precisely, we obtain Laurent series of the following type.

\begin{enumerate}
\item
If $0\leq{}x<1$, then the zero-order term of~$\left[x\right]_q$ vanishes:
$$
\left[x\right]_q=\varkappa_1q+\varkappa_2q^2+\varkappa_3q^3+\cdots
$$

\item
If $-1\leq{}x<0$, then 
$$
\left[x\right]_q=-\frac{1}{q}+\varkappa_0+\varkappa_1q+\varkappa_2q^2+\varkappa_3q^3+\cdots
$$

\item
More generally, applying the operator~$T^{-1}$ several times, we obtain the general
form \eqref{GeneralEq} of~$\left[x\right]_q$ in the case $-k\leq{}x<1-k$.
\end{enumerate}

\begin{ex}
(a) If $n$ is a positive integer, then
$\left[-n\right]_q=q^{-n}+q^{1-n}+\cdots+q^{-1}$.

(b) For $x=\frac{1}{2}$ we have $\left[\frac{1}{2}\right]_q=\frac{q}{1+q}$, and for $x=-\frac{1}{2}$ we have $\left[-\frac{1}{2}\right]_q=-\frac{1}{q(1+q)}$,
so that 
$$
\left[\frac{1}{2}\right]_q+\left[-\frac{1}{2}\right]_q=-\frac{1}{q}+1.
$$
\end{ex}

We calculated the series corresponding to the negations of several examples of quadratic irrationals, 
and observed the following property.
Some of them satisfy the property that the sum of the series $\left[x\right]_q+\left[-x\right]_q$ contains only finitely many terms.

\begin{ex}
\label{UltimEx}
(a) For $-\sqrt{2}$, we have
$$
\left[-\sqrt{2}\right]_q=-\frac{1}{q^2}-1+q-q^{3}+2q^{5}
-q^6- 4q^7+5q^8+7q^9-18q^{10} - 7q^{11}+ 55q^{12}- 18q^{13}\cdots
$$
Starting from the third-order term, this series is the negation of~$\left[\sqrt{2}\right]_q$.

(b)  For $-\sqrt{7}$, we have
$$
\left[-\sqrt{7}\right]_q=-\frac{1}{q^3}-\frac{1}{q^2}-1+q^2-q^3+q^4-2q^5
+3q^6-4q^7+6q^8-8q^9+9q^{10}-9q^{11}+5q^{12}+9q^{13}\cdots
$$
Once again, starting from the third-order term, the series is the negation of~$\left[\sqrt{7}\right]_q$.
\end{ex}

This property is a mystery to us, since, as we can see already for $x=\frac{7}{5}$, it fails to be true even for rationals.

Note also, that nothing similar to Example~\ref{UltimEx} happens for more sophisticated irrationals.

\begin{ex}
\label{UltimExBis}
For $-\pi$, we have the series that starts as follows
$$
\begin{array}{rcl}
\left[-\pi\right]_q&=&-\frac{1}{q^4}-\frac{1}{q^2}-\frac{1}{q}-q^{3}+q^{4}
-q^{10}+q^{11}-q^{17}+2q^{18}-3q^{19}+3q^{20}-q^{21}-q^{24}+3q^{25}-6q^{26}\\[4pt]
&&+6q^{27}-2q^{28}-q^{31}+4q^{32}-9q^{33}+10q^{34}-6q^{35}+3q^{36}-q^{37}-q^{38}+5q^{39}-13q^{40}\cdots
\end{array}
$$
and seems to have nothing in common with~$\left[\pi\right]_q$.
\end{ex}

\bigbreak \noindent
{\bf Acknowledgements}.
We are grateful to Vladlen Timorin for
a suggestion to look at Taylor coefficients.
We are also grateful to Fr\'ed\'eric Chapoton for a SAGE program computing $q$-rationals and to
Doron Zeilberger for the proof of the functional equations formulated as conjectures in the first version of this paper.
This paper was partially supported by the ANR project SC3A, ANR-15-CE40-0004-01.


\begin{thebibliography}{99}

\bibitem{RR}
B. Berndt, H. Chan, S-S. Huang, S-Y. Kang, J. Sohn, S. Son, 
{\it The Rogers-Ramanujan continued fraction}, 
J. Comput. Appl. Math. {\bf 105} (1999), no. 1-2, 9--24. 

\bibitem{Never}
J. Borwein, A. van der Poorten, J. Shallit, W. Zudilin, 
Neverending fractions. 
An introduction to continued fractions. 
Australian Mathematical Society Lecture Series, 23. Cambridge University Press, Cambridge, 2014.

\bibitem{EYZ}
S.B. Ekhad, M. Yang, D. Zeilberger,
{\it Automated Proofs of Many Conjectured Recurrences in the OEIS made by R.J. Mathar},
arXiv:1707.04654.

 \bibitem{HW}
G. H. Hardy, E. M. Wright,  
An introduction to the theory of numbers. 
Sixth edition. Revised by D. R. Heath-Brown and J. H. Silverman. With a foreword by Andrew Wiles. 
Oxford University Press, Oxford, 2008, 621~pp.

\bibitem{MGO}
S.~Morier-Genoud, V.~Ovsienko, 
{\it Farey boat.
Continued fractions and triangulations,
modular group and polygon dissections},
 Jahresber. Dtsch. Math.-Ver. {\bf 121} (2019), no. 2, 91--136.

\bibitem{MGOqR} 
S.~Morier-Genoud, V.~Ovsienko, 
{\it $q$-deformed rationals and $q$-continued fractions},
arXiv:1812.00170.

\bibitem{OEIS} 
OEIS Foundation Inc., The On-Line Encyclopedia of Integer Sequences, http://oeis.org.

\bibitem{Sta}
R. Stanley, 
Enumerative combinatorics. Volume 1. Second edition. 
Cambridge Studies in Advanced Mathematics, 49. 
Cambridge University Press, Cambridge, 2012. xiv+626 pp.

\bibitem{Zei}
D. Zeilberger, {\it The C-finite ansatz}, Ramanujan J. {\bf 31} (2013), no. 1-2, 23--3.


\end{thebibliography}
\end{document}